\documentclass{amsart}

\newtheorem{theorem}{Theorem}[section]
\newtheorem{lemma}[theorem]{Lemma}
\newtheorem{proposition}[theorem]{Proposition}
\newtheorem{corollary}[theorem]{Corollary}

\theoremstyle{definition}

\theoremstyle{remark}

\numberwithin{equation}{section}

\begin{document}

\title{ Gradient Shrinking Solitons with Vanishing Weyl Tensor}

\author{Zhu-Hong Zhang}
\address{Department of Mathematics, Sun Yat-Sen University, Guangzhou, P.R.China 510275}
\email{juhoncheung@sina.com}
\thanks{}



\date{July 10, 2008.}



\begin{abstract}
In this paper, we will give a local version of the Hamilton-Ivey
type pinching estimate of the gradient shrinking soliton with
vanishing Weyl tensor, and then give a complete classification on
gradient shrinking solitons with vanishing Weyl tensor.
\end{abstract}

\maketitle



\section{Introduction}

Let $(M,g,f)$ be a gradient shrinking soliton, i.e., $(M,g)$ is a
smooth Riemannian manifold with a smooth function $f$, and satisfies
$R_{ij}+\nabla_i\nabla_j f=\lambda g_{ij}$, where $\lambda$ is a
positive constant.

The classification of the complete gradient shrinking soliton is an
important problem in the theory of the Ricci flow. Recently there
are some work which deal with this problem. Note that people often
do not distinguish between the gradient shrinking soliton and the
self-similar solution (see chapter 2 of \cite{{CK}} for the
definition). In fact, the author (see theorem 1 of \cite{ZZH}) had
shown that the complete gradient shrinking soltion must be the
self-similar solution, so they are the same.

In dimension 2 Hamilton \cite{Ha95} proved that any 2-dimensional
complete non-flat ancient solution of bounded curvature must be
$S^2$, $RP^2$, or the cigar soliton. In dimension 3, Ivey
\cite{Ivey} first showed that the compact 3-dimensional gradient
shrinking soliton has constant positive sectional curvature. For the
noncompact case, Perelman \cite{Pere} had shown that 3-dimensional
complete non-flat gradient shrinking soliton with bounded and
nonnegative sectional curvature, and in addition, is
$\kappa$-noncollapsed on all scales, must be the finite quotients of
$S^2\times R$, or $S^3$. This result of Perelamn had been improved
by Ni-Wallach \cite{NW} and Naber \cite{Na} in which they dropped
the assumption on $\kappa$-noncollapsing condition and replaced
nonnegative sectional curvature by nonnegative Ricci curvature. In
addition, Ni-Wallach \cite{NW} allowed the curvature to grow as fast
as $e^{ar(x)^2}$, where $r(x)$ is the distance function and $a$ is a
suitable small positive constant. In particular, Ni-Wallach's result
implied that any 3-dimensional complete noncompact non-flat gradient
shrinking soliton with nonnegative Ricci curvature and with
curvature not growing faster than $e^{ar(x)^2}$ must be a quotient
of the round infinite cylinder $S^2\times R$. Recently, by a local
version of the Hamilton-Ivey pinching estimate due to Chen
(Corollary 2.4 of \cite{Chen}), Cao-Chen-Zhu \cite{CCZ} obtained the
following complete classification (without any curvature bound
assumption) on 3-dimensional complete gradient shrinking soliton

\begin{theorem}
(the proposition 4.7 of \cite{CCZ}) The only 3-dimensional complete
gradient shrinking solitons are the finite quotients of $R^3$,
$S^2\times R$, and $S^3$.
\end{theorem}

The classification of complete gradient shrinking solitons on high
dimension is more difficult. Note that the 3-dimensional manifold
automatically has vanishing Weyl tensor. There are some recent work
on the complete gradient shrinking solitons with vanishing Weyl
tensor on high dimension. The first classification theorem with the
dimension $n\ge 4$ given by Gu-Zhu (see proposition 4.1 of
\cite{GZ}) says that any non-flat, $\kappa$-noncollapsing,
rotationally symmetric gradient shrinking soliton with bounded and
nonnegative sectional curvature must be the finite quotients of
$S^n\times R$ or $S^{n+1}$. Later, Kotschwar \cite{K}, improved this
result, showed that any complete rotationally symmetric gradient
shrinking soliton (in which not any curvature bound assumption) is
the finite quotients of $R^{n+1}$, $S^n\times R$, or $S^{n+1}$. Note
that any rotationally symmetric metric has vanishing Weyl tensor. In
the more general case of vanishing Weyl tensor, Ni-Wallach \cite{NW}
showed that a complete n-dimensional gradient shrinking soliton with
vanishing Weyl tensor and assume further that, it has nonnegative
Ricci curvature and the growth of the Riemannian curvature is not
faster that $e^{a(r(x)+1)}$, where $r(x)$ is the distance function
and $a$ is a suitable positive constant, then its universal cover is
either $R^n$, $S^n$, or $S^{n-1}\times R$. This result had been
improved by Peterson-Wylie (Theorem 1.2 and the remark 1.3 of
\cite{PW}) in which they only need to assume the Ricci curvature is
bounded from below and the growth of the Ricci curvature is not
faster than $e^{\frac{2}{5}cr(x)^2}$ outside of a compact set, where
$c<\frac{\lambda}{2}$. We also notice that Cao-Wang \cite{CW} had an
alternative proof of the Ni-Wallach's result \cite{NW}.

The key point to get the above complete classification theorem on
3-dimensional complete gradient shrinking soliton without curvature
bound assumption is the local version of the Hamilton-Ivey pinching
estimate. The Hamilton-Ivey pinching estimate in 3-dimension plays a
crucial role in the analysis of the Ricci flow. An open question is
how to generalize Hamilton-Ivey's to high dimension. In \cite{ZZH2},
the author obtained the following (global) Hamilton-Ivey type
pinching estimate on high dimension: \emph{Suppose we have a
solution to the Ricci flow on an n-dimension manifold which is
complete with bounded curvature and vanishing Weyl tensor for each
$t\ge0$. Assume at $t=0$ the least eigenvalue of the curvature
operator at each point are bounded below by $\nu\ge-1$. Then at all
points and all times $t\ge0$ we have the pinching estimate
$$R\ge(-\nu)[\log(-\nu)+\log(1+t)-\frac{n(n+1)}{2}]
$$ whenever $\nu<0$.} In the present paper, we will get a local version of
this Hamilton-Ivey type pinching estimate for the gradient shrinking
solitons with vanishing Weyl tensor (without curvature bound). The
idea is to use the methods of \cite{Chen}. But since the curvature
operator is more complicated on high dimension, we firstly need to
prove an algebra lemma. Based on this pinching estimate, we will
obtain the following complete classification theorem (without any
curvature bound assumption)

\begin{theorem}
Any complete gradient shrinking soliton with vanishing Weyl tensor
must be the finite quotients of $R^n$, $S^{n-1}\times R$, or $S^n$.
\end{theorem}

This paper contains three sections and the organization is as
follows. In section 2, we will prove an algebra lemma which will be
used to prove the local version of the Hamilton-Ivey type pinching
estimate. In section 3, we will give some propositions and finish
the proof of the theorem 1.2.

\vskip 0.3cm \noindent {\bf Acknowledgement}  I would like to thank
my advisor Professor X.P.Zhu for his encouragement and many helpful
suggestions. I would also like to thank Professor B.L.Chen for
helpful discussions.

\section{An Algebra Lemma}

Let $x_1,\ldots,x_n$ be any $n(\ge 4)$ real numbers, $m$ be any
positive integer, and denote by
$M_{ij}\stackrel{\mathrm{\Delta}}{=}x_i+x_j$. Define a function as
follow
$$\arraycolsep=1.5pt\begin{array}{rcl}
f(x_1,\cdots,x_n)
\stackrel{\mathrm{\Delta}}{=}&&-\frac{\sum\limits_{i<j,\{i,j\}\ne\{1,2\}}M_{ij}}{m+1}
                                                 \bigg[\sum\limits_{i<j \atop \{i,j\}\ne\{1,2\}}M_{ij}+(m+1)M_{12}\bigg]
                                                 -M_{12}\sum\limits_{i<j \atop \{i,j\} \ne \{1,2\}}M_{ij}    \\[4mm]
                  &+& \bigg[\sum\limits_{i<j \atop \{i,j\}\ne\{1,2\}}(M_{ij}^2+\sum\limits_{k\ne i,j} M_{ik}M_{jk})
                  +(m+1)\sum\limits_{k\ne 1,2} M_{1k}M_{2k}\bigg] .  \\[4mm]
\end{array}$$

\begin{lemma}
Let $\rho$ be a nonnegative constant,
 then there exist a nonnegative constant $C(m,n)$, such that $f \ge -C(m,n)\rho^2$ under the following
 assumptions :\\
 (i) $x_1 \le x_2 \le \min\limits_{3 \le i \le n}x_i $;\\
 (ii) $\sum\limits_{i<j \atop \{i,j\}\ne\{1,2\}}M_{ij}+m M_{12}\ge -\rho
 $;\\
 (iii) $\sum\limits_{i<j \atop \{i,j\}\ne\{1,2\}}M_{ij}+(m+1)M_{12}< -(m+1)(m+n-1)\rho $.
\end{lemma}

\begin{proof}
We firstly claim that $f(\ldots,x_3,x_4,\ldots)\ge
f(\ldots,x,x,\ldots)$, where $2x=x_3+x_4$.

Indeed, without loss of generality, we can assume $x_3<x_4$. Let
$2\delta=x_4-x_3$, then by direct computation we get ($M_0=x$)
$$\arraycolsep=1.5pt\begin{array}{rcl}
&&f(\ldots,x_3,x_4,\ldots)- f(\ldots,x,x,\ldots)\\[4mm]
&&\hskip 1cm=\sum\limits_{k\ne3,4}(M_{3k}M_{4k}-M_{0k}^2)
             +\sum\limits_{i<j,\{i,j\} \ne \{1,2\} \atop \{i,j\}\ne\{3,4\}}(M_{3i}M_{3j}+M_{4i}M_{4j}-2M_{0i}M_{0j}))
           \\[4mm]
&&\hskip 2cm+\sum\limits_{i \ne 3,4}(M_{3i}^2+\sum\limits_{k
                           \ne3,i}M_{3k}M_{ik}-M_{0i}^2-\sum\limits_{k \ne3,i}M_{0k}M_{ik})\\[4mm]
&&\hskip 2cm+\sum\limits_{i \ne 3,4}(M_{4i}^2+\sum\limits_{k
                           \ne4,i}M_{4k}M_{ik}-M_{0i}^2-\sum\limits_{k \ne4,i}M_{0k}M_{ik})\\[4mm]
&&\hskip 1.5cm+(m+1)(M_{13}M_{23}+M_{14}M_{24}-2M_{10}M_{20})\\[4mm]
&&\hskip 1cm=\sum\limits_{k\ne3,4}(-\delta^2)
                   +\sum\limits_{i<j,\{i,j\} \ne \{1,2\} \atop \{i,j\}\ne\{3,4\}}2\delta^2
                   +\sum\limits_{i \ne 3,4}2\delta^2
                   +(m+1)2\delta^2\\[4mm]
&&\hskip 1cm\geq 0.
\end{array}$$

 Note that $f$ and
the assumptions are symmetric with respect to $x_3,\ldots,x_n$. By
the above claim, we only need to prove the special case that
$x_3=\cdots=x_n$. In this case,
$$\sum\limits_{i<j \atop \{i,j\}\ne\{1,2\}}M_{ij}=(n-2)\big(M_{12}+\frac{n-1}{2}M_{33}\big),$$

and
$$\arraycolsep=1.5pt\begin{array}{rcl}
 f(x_1, x_2, x_3, \ldots , x_3)
  &=& -\frac{(n-2)\big(M_{12}+\frac{n-1}{2}M_{33}\big)}{m+1}
                                                 \bigg[\sum\limits_{i<j \atop \{i,j\} \ne \{1,2\}}M_{ij}+(m+1)M_{12}\bigg] \\[4mm]
  &&+\frac{n-2}{2}\Big[(n-1)M_{13}^2+(n-1)M_{23}^2+(n-1)(n-3)M_{33}^2\\[4mm]
  &&\hskip 1.2cm +(n-3)M_{12}M_{33}+2(m+1)M_{13}M_{23}\Big] \\[4mm]
  &=& I+II. \\[4mm]
\end{array}$$

Clearly, $-(m+1)(m+n-1)\rho\le -2\rho\le0$, and we have some
estimates in the following assertion.

\vskip 0.1cm \noindent{\bf Claim} \emph{The following inequalities
hold\\
(1) $M_{12}<-\rho\le0$;\\
(2) $M_{33}>0$;\\
(3) $M_{12}+\frac{n-1}{2}M_{33}>0$;\\
(4) $(m+n-1)(-M_{12})\ge \frac{(n-1)(n-2)}{2}M_{33}$;\\
(5) $\frac{(n-1)(n-2)}{2}M_{33}\ge -\rho-(m+n-2)M_{12}$;\\
(6) $(n-2)\big(M_{12}+\frac{n-1}{2}M_{33}\big)\ge(m-1)(-M_{12})$.}

Obviously, by combining the assumptions (ii) and (iii), we get
(1)-(3).

Then by (iii), we have
$(n-2)\big(M_{12}+\frac{n-1}{2}M_{33}\big)+(m+1)M_{12}\le0$, so we
get (4).

Now by (ii), $(n-2)\big(M_{12}+\frac{n-1}{2}M_{33}\big) \ge
-\rho-mM_{12}(\ge 0)$, we have (5).

(6) follows from (5) immediately.

And we have proved the assertion.

By the above claim we know that $I$ is always nonnegative. In the
following, we will divede the argument into two cases to consider.

Case 1) $m=1, 2$. Then we have
$$\arraycolsep=1.5pt\begin{array}{rcl}
II &\ge& \frac{n-2}{2}\bigg\{3M_{13}^2+3M_{23}^2-6|M_{13}M_{23}| \\[4mm]
&&\hskip
1cm+(n-3)\Big[M_{12}+\frac{n-1}{2}M_{33}\Big]M_{33}+\frac{(n-1)(n-3)}{2}M_{33}^2
\bigg\}\\[4mm]
&\ge&0 \\[4mm]
\end{array}$$
In this case, Lemma 2.1 is proved.

Case 2) $m \ge 3$. It suffices to prove that $M_{13}M_{23}< 0$, that
is to say, $M_{13}<0,\ M_{23}>0$, which implies $-M_{12}M_{33}\ge
-M_{13}M_{23}>0$. (Indeed, if $M_{13}M_{23}\ge 0$, it is easy to see
that II is positive , therefore we have proved the lemma 2.1. )

then we have
$$\arraycolsep=1.5pt\begin{array}{rcl}
I &\ge& \frac{(m-1)(-M_{12})}{m+1}(m+1)(m+n-1)\rho \\[4mm]
  &\ge& (m-1)\frac{(n-1)(n-2)}{2}\rho M_{33} \\[4mm]
  &\ge& (n-3)\rho M_{33}, \\[4mm]
II &\ge&\frac{n-2}{2}\Big\{(n-1)M_{13}^2+(n-1)M_{23}^2+2(m+1)M_{13}M_{23}  \\[4mm]
    &&\hskip 1.1cm +\frac{n-3}{n-2}M_{33}[-\rho-mM_{12}]+\frac{n-3}{n-2}M_{33}[-\rho-(m+n-2)M_{12}] \Big\} \\[4mm]
    &\ge& \frac{n-2}{2}\Big\{(n-1)M_{13}^2+(n-1)M_{23}^2+2(m+1)M_{13}M_{23} \\[4mm]
    &&\hskip 1.1cm-2(m+1)\frac{n-3}{n-2}M_{12}M_{33}-2\frac{n-3}{n-2}\rho M_{33}\Big\}  \\[4mm]
    &\ge&\frac{n-2}{2}\Big\{(n-1)M_{13}^2+(n-1)M_{23}^2+\frac{2(m+1)}{n-2}M_{13}M_{23} \Big\} \\[4mm]
    &&\hskip 0.2cm-(n-3)\rho M_{33}. \\[4mm]
\end{array}$$
Therefore
$f\ge\frac{(n-1)(n-2)}{2}\Big[M_{13}^2+M_{23}^2+2\frac{m+1}{(n-1)(n-2)}M_{13}M_{23}\Big]$.

If  $\frac{m+1}{(n-1)(n-2)}\le 1$, then $f\ge0$.

If  $\frac{m+1}{(n-1)(n-2)}> 1$, and
$M_{23}+2\frac{m+1}{(n-1)(n-2)}M_{13}\ge0$, we also have $f\ge0$.

Otherwise, $m+1>(n-1)(n-2)$, and
$M_{13}<-\frac{(n-1)(n-2)}{2(m+1)}M_{23}$.

Since
$M_{12}+\frac{n-1}{2}M_{33}=M_{13}+M_{23}+\frac{n-3}{2}M_{33}$, by
(ii), we get
$$\frac{(n-1)(n-2)}{2} \frac{2(m+1)-(n-1)(n-2)}{2(m+1)}M_{23}+[m-\frac{(n-2)(n-3)}{2}]M_{12}\ge-\rho,$$
So
$$\arraycolsep=1.5pt\begin{array}{rcl}
  -M_{12} &\le&
\frac{(n-1)(n-2)}{2(m+1)}\frac{2(m+1)-(n-1)(n-2)}{2m-(n-2)(n-3)}M_{23}+\frac{2}{2m-(n-2)(n-3)}\rho \\[4mm]
&\le&\frac{(n-1)(n-2)}{2(m+1)}M_{23}+\frac{1}{n}\rho ,\\[4mm]
\end{array}$$
and then
$$\arraycolsep=1.5pt\begin{array}{rcl}
f&\ge&\frac{(n-1)(n-2)}{2}\Big\{[\frac{(n-1)(n-2)}{2(m+1)}M_{23}]^2-\frac{2(m+1)}{n(n-1)(n-2)}\rho M_{23}\Big\}\\[4mm]
&\ge&-C(m,n)\rho^2 .\\[4mm]
\end{array}$$
where $C(m,n)$ is a constant depending only on $m$ and $n$.

Therefore we have proved the Lemma.
\end{proof}

\section{The Proof of the Main Theorem}

Let $(M,g,f)$ be a smooth gradient shrinking soliton, then by using
the contracted second Bianchi identity we get the equation
$R+|{\nabla}f|^2-2{\lambda}f=const$. Obviously, by rescaling $g$ and
changing $f$ by a constant we can assume $\lambda=\frac{1}{2}$ and
$R+|{\nabla}f|^2-f=0$. We call such a soliton normalized, and $f$ a
normalized soliton function.

In terms of moving frames \cite{Ha86} of the Ricci flow, the
curvature operator $M_{\alpha\beta}$ has the following evolution
equation
$$\frac{\partial}{\partial t}M_{\alpha\beta}=\triangle
M_{\alpha\beta}+M_{\alpha\beta}^2+M_{\alpha\beta}^{\#} ,$$ where
$M_{\alpha\beta}^{\#}$ is the lie algebra adjoint of
$M_{\alpha\beta}$. In general, we know little about
$M_{\alpha\beta}^{\#}$. However, when the metric is conformally
flat, we have the following proposition.

\begin{proposition}
Suppose we have a smooth solution $g_{ij}(x,t)$ of the Ricci flow on
an $n$-dimensional manifold $M$, and suppose at $t=t_0$ the metric
$g_{ij}(x,t_0)$ has vanishing Weyl tensor. Then at $t=t_0$, for any
point $p$, there exist an orthonormal basis $\{e_i\}$ and $n$ real
numbers $M_i$ such that $\{\phi^{\alpha}=\sqrt{2}e_i\wedge
e_j\}(i<j)$ is an orthonormal basis of $\wedge^2 T_p{M}$, and we
have

(i) $M_{\alpha\beta}=
      \left\{
       \begin{array}{ll}
         M_{ij}\stackrel{\mathrm{\Delta}}{=}M_i+M_j\ ,\qquad &if\ \phi^{\alpha}=\phi^{\beta}=\sqrt{2}e_i\wedge e_j ; \\[4mm]
         0\ ,\qquad &if\ \alpha\neq\beta .
       \end{array}
    \right.
    $

(ii) $M_{\alpha\beta}^{\#}=
      \left\{
       \begin{array}{ll}
         \sum\limits_{k\ne i,j} M_{ik}M_{jk}\ ,\qquad &if\ \phi^{\alpha}=\phi^{\beta}=\sqrt{2}e_i\wedge e_j ; \\[4mm]
         0\ ,\qquad &if\ \alpha\neq\beta .
       \end{array}
    \right.
    $
\end{proposition}

\begin{proof}
Indeed, suppose $\{e_i\}$ is an orthonormal basis, which
diagonalizes Ricci tensor, i.e., $Ric(e_i)=\lambda_i e_i$.

Because the Weyl tensor vanish, we have
$$R_{ijkl}=\frac{1}{n-2}(R_{ik}g_{jl}+R_{jl}g_{ik}-R_{il}g_{jk}-R_{jk}g_{il})-\frac{1}{(n-1)(n-2)}R(g_{ik}g_{jl}-g_{il}g_{jk}).$$
Thus
$R_{ijij}=\frac{\lambda_i+\lambda_j}{n-2}-\frac{1}{(n-1)(n-2)}R$,
and $R_{ijkl}=0$ if three of the indices are mutually distinct. Let
$M_i=\frac{2\lambda_i}{n-2}-\frac{1}{(n-1)(n-2)}R$, and we have
proved (i).

Note that
$M_{\alpha\beta}^{\#}=C_{\alpha}^{\gamma\eta}C_{\beta}^{\delta\theta}M_{\gamma\delta}M_{\eta\theta}=C_{\alpha}^{\gamma\eta}C_{\beta}^{\gamma\eta}M_{\gamma\gamma}M_{\eta\eta}$,
where $[\phi^{\alpha},
\phi^{\beta}]=C_{\gamma}^{\alpha\beta}\phi^{\gamma}$.

Let $A_{ij}(i\neq j)$ denote by the matric with $(A_{ij})_{ij}=1,\
(A_{ij})_{ji}=-1$, and any other elements are 0. Then $[A_{ij},
A_{jk}]=A_{ik}$, if $i<j<k$.

By direct computation, we have $M_{\alpha\beta}^{\#}=0$, if
$\alpha\neq\beta$. And If $\alpha=\beta=\sqrt{2}e_i\wedge e_j(i<j)$,
we have
$$\arraycolsep=1.5pt\begin{array}{rcl}
    &&M_{\alpha\alpha}^{\#}
    = (C_{\alpha}^{\gamma\delta})^2 M_{\gamma\gamma}M_{\delta\delta}
    =<[\frac{\sqrt{2}}{2}A_{ij}, \phi^{\gamma}], \phi^{\delta}>^2M_{\gamma\gamma}M_{\delta\delta} \\[4mm]
    &=& \sum\limits_{k\neq i,j}{<[\frac{\sqrt{2}}{2}A_{ij}, \frac{\sqrt{2}}{2}A_{ki}], \phi^{\delta}>^2M_{ki}M_{\delta\delta}}
    +\sum\limits_{k\neq i,j}{<[\frac{\sqrt{2}}{2}A_{ij}, \frac{\sqrt{2}}{2}A_{kj}], \phi^{\delta}>^2M_{kj}M_{\delta\delta}} \\[4mm]
    &=& \frac{1}{2}\sum\limits_{k\neq i,j}{<\frac{\sqrt{2}}{2}A_{jk}, \phi^{\delta}>^2M_{ki}M_{\delta\delta}}
    +\frac{1}{2}\sum\limits_{k\neq i,j}{<\frac{\sqrt{2}}{2}A_{ik}, \phi^{\delta}>^2M_{kj}M_{\delta\delta}} \\[4mm]
    &=& \sum\limits_{k\ne i,j}M_{ik}M_{jk}.
\end{array}$$

\end{proof}

Now, combing Lemma 2.1 and Porposition 3.1, we are ready to prove
the local version of the Hamilton-Ivey type pinching estimate. The
basic idea is to use the methods of \cite{Chen}.

\begin{proposition}
For any nonnegative integer $m$, there is a constant $C_{m}$
depending only on $m$ and $n$ satisfying the following property.
Suppose we have a complete gradient shrinking soliton $(M^n,
g_{ij}(x,t))(n\ge4)$ on $[0, T]$ with vanishing Weyl tensor. And
assume that $Ric(x,t)\le(n-1)r_0^{-2}$ for $x\in B_t(x_0,Ar_0)$,
$t\in[0,T]$ and $R+m\nu\ge-K_m(K_m\ge0)$ on $B_0(x_0,Ar_0)$ at
$t=0$, where $\nu$ denote by the least eigenvalue of the curvature
operator. Then we have
$$R(x,t)+m\nu\ge min\{-\frac{C_m}{t+\frac{1}{K_m}}, -\frac{C_m}{Ar_0^2}\}, if A\ge2 ,$$
whenever $x\in B_t(x_0, \frac{1}{2}Ar_0)$, $t\in[0,T]$.
\end{proposition}

\begin{proof}
By \cite{Pere02}, we have
$$(\frac{\partial}{\partial
t}-\Delta)d_t(x_0,x)\ge-\frac{5(n-1)}{3}r_0^{-1},$$ whenever
$d_t(x_0,x)>r_0$, in the sense of support function.

We will argue by induction on $m$ to prove the estimate holds on
ball of radius $(\frac{1}{2}+\frac{1}{2^{m+2}})Ar_0$. The case $m=0$
follows from Theorem 1 of \cite{ZZH}. Suppose we have proved the
result for some $m_0(\ge 0)$, that is to say, there is a constant
$C_{m_0}$ such that
$$R(x,t)+m_0\nu\ge \min\{-\frac{C_{m_0}}{t+\frac{1}{K_{m_0}}}, -\frac{C_{m_0}}{Ar_0^2}\} ,$$
whenever $x\in B_t(x_0, (\frac{1}{2}+\frac{1}{2^{m_0+2}})Ar_0)$,
$t\in[0,T]$. We are going to prove the result for $m=m_0+1$ on ball
of radius $(\frac{1}{2}+\frac{1}{2^{m_0+3}})Ar_0$.

Without loss of generality, we may assume $K_0\le K_1\le K_2\le
\cdots$.

Define a function
$C_{m_0}(t)\stackrel{\mathrm{\Delta}}{=}\max\{\frac{C_{m_0}}{t+\frac{1}{K_{m_0}}},
\frac{C_{m_0}}{Ar_0^2}\}$.

Under the moving frame, let
$$N_{\alpha\beta}\stackrel{\mathrm{\Delta}}{=}Rg_{\alpha\beta}+(m_0+1)M_{\alpha\beta},\qquad P_{\alpha\beta}\stackrel{\mathrm{\Delta}}{=}\varphi(\frac{d_t(x,x_0)}{Ar_0})N_{\alpha\beta} ,$$
where $\varphi$ is a smooth nonnegative decreasing function, which
is 1 on $(-\infty, \frac{1}{2}+\frac{1}{2^{m_0+3}}]$, and 0 on
$[\frac{1}{2}+\frac{1}{2^{m_0+2}}, \infty )$.

By direct computation, we have
$$(\frac{\partial}{\partial t}-\Delta)P_{\alpha\beta}=-2\nabla_l \varphi\nabla_l N_{\alpha\beta}+Q_{\alpha\beta} ,$$
where $Q_{\alpha\beta}$ is
$$\arraycolsep=1.5pt\begin{array}{rcl}
Q_{\alpha\beta}
    &=& \Big[\varphi'\frac{1}{Ar_0}(\frac{\partial}{\partial t}-\Delta)d_t-\varphi''\frac{1}{(Ar_0)^2} \Big]N_{\alpha\beta} \\[4mm]
    &&\hskip 0.1cm +\varphi\Big[g_{\alpha\beta}(\frac{\partial}{\partial t}-\Delta)R+(m_0+1)(\frac{\partial}{\partial t}-\Delta)M_{\alpha\beta}  \Big]\\[4mm]
    &=& \Big[\varphi'\frac{1}{Ar_0}(\frac{\partial}{\partial t}-\Delta)d_t-\varphi''\frac{1}{(Ar_0)^2} \Big]N_{\alpha\beta} \\[4mm]
    &&\hskip 0.1cm +\varphi g_{\alpha\beta}\Big[\sum\limits_{i<j}(M_{ij}^2+\sum\limits_{k\neq i,j}M_{ik}M_{jk})+(m_0+1)(M_{i_0j_0}^2+\sum\limits_{k\neq i_0,j_0}M_{i_0k}M_{j_0k})  \Big],\\[4mm]
\end{array}$$
if $\phi^{\alpha}=\sqrt{2}e_{i_0}\wedge e_{j_0}$ and the second
equality follows from the proposition 3.1.

Note that the least eigenvalue of $N_{\alpha\beta}$ is
$R+(m_0+1)\nu$.

Let$$u(t)\stackrel{\mathrm{\Delta}}{=}\min\limits_{x\in M}[R+(m_0+1)
\nu]\varphi(x,t) .$$ For fixed $t_0\in [0, T]$, assume $[R+(m_0+1)
\nu]\varphi(x_0,t_0)=u(t_0)<-(m_0+2+1)(m_0+2+n-1)C_{m_0}(t_0)$.
Otherwise, we have the estimate at time $t_0$.

 Let V be the corresponding unit eigenvector of the least eigenvalue of $N_{\alpha\beta}$
at $(x_0,t_0)$. Suppose the eigenvalue of the Ricci tensor is
$\{\lambda_i\}$, where $\lambda_1 \le \lambda_2 \le \cdots \le
\lambda_n$. Then by Lemma 2.1 and Proposition 3.1, we have
$$\arraycolsep=1.5pt\begin{array}{rcl}
Q(V,V)(x_0,t_0)
    &=& \Big[\varphi'\frac{1}{Ar_0}(\frac{\partial}{\partial t}-\Delta)d_t-\varphi''\frac{1}{(Ar_0)^2} \Big]\frac{u(t_0)}{\varphi} \\[4mm]
    &&\hskip 0.1cm +\varphi \Big[\sum\limits_{i<j}(M_{ij}^2+\sum\limits_{k\neq i,j}M_{ik}M_{jk})+(m_0+1)(M_{12}^2+\sum\limits_{k\neq 1,2}M_{1k}M_{2k})  \Big]\\[4mm]
    &=& \Big[\varphi'\frac{1}{Ar_0}(\frac{\partial}{\partial t}-\Delta)d_t-\varphi''\frac{1}{(Ar_0)^2} \Big]\frac{u(t_0)}{\varphi} \\[4mm]
    &&\hskip 0.1cm +\varphi \Big[\frac{\Big(\sum\limits_{i<j \atop \{i,j\} \ne \{1,2\}}M_{ij}+(m_0+2)M_{12}\Big)^2}{m_0+2}+f(M_1,\cdots,M_n)  \Big](here\ m=m_0+1) \\[4mm]
    &\ge& \Bigg[\varphi'\frac{1}{Ar_0}(-\frac{5(n-1)}{3}r_0^{-1})-\varphi''\frac{1}{(Ar_0)^2} \Bigg]\frac{u(t_0)}{\varphi} \\[4mm]
    &&\hskip 0.1cm +\frac{1}{(m_0+2)\varphi}u(t_0)^2-\varphi
    C(m_0)C_{m_0}(t_0)^2  \\[4mm]
    &=& \frac{1}{(m_0+2)\varphi}\Big[u(t_0)^2-\Big(\varphi'\frac{1}{Ar_0^2}\frac{5(n-1)(m_0+2)}{3}+\varphi''\frac{m_0+2}{(Ar_0)^2}\Big)u(t_0)\Big] \\[4mm]
    &&\hskip 0.1cm-C(m_0)C_{m_0}(t_0)^2. \\[4mm]
\end{array}$$

Since $|\varphi '|\le C(m_0),|\varphi ''|\le C(m_0),\frac{(\varphi
')^2}{\varphi}\le C(m_0)$, by applying maximum principle, we have
$$\arraycolsep=1.5pt\begin{array}{rcl}
\left.\frac{d^{-}}{dt}u\right|_{t=t_0}
    &\ge&  Q(V,V)(x_0,t_0)+\frac{2}{(Ar_0)^2}\frac{(\varphi ')^2}{\varphi^2}u(t_0) \\[4mm]
    &\ge&  \frac{1}{2(m_0+2)}u(t_0)^2, \\[4mm]
\end{array}$$
provided $|u|(t_0)\ge \max\Big\{C(m_0)C_{m_0}(t_0),
\frac{C(m_0)}{Ar_0^2}\Big\}$. By integrating the above differential
inequality, we get the estimate
$$u(t)\ge \min\Big\{\frac{1}{\frac{1}{u(0)}-\frac{t}{2(m_0+2)}}, -C(m_0)C_{m_0}(t), -\frac{C(m_0)}{Ar_0^2}\Big\}.$$

Clearly, there is a constant $C_{m_0+1}$, s.t. $u(t)\ge
\min\{-\frac{C_{m_0+1}}{t+\frac{1}{K_{m_0+1}}},
-\frac{C_{m_0+1}}{Ar_0^2}\}$ .

\end{proof}

\vskip 0.1cm\noindent{\bf Remark} In fact, the case $m=0$, we do not
need to suppose the soliton has vanishing Weyl tensor, since Chen in
\cite{Chen} have proved this result.

\begin{corollary}
Any gradient shrinking soliton(not necessarily having bounded
curvature) with vanishing Weyl tensor must have nonnegative
curvature operator.
\end{corollary}

\begin{proof}
Let $(M,g,f)$ be a gradient shrinking soliton, then we have a
solution $g(t)(t\in(-\infty,0])$ of the Ricci flow with $g(0)=g$.

The case $n=3$ has done by Chen \cite{Chen}. Therefore we only need
to prove the case $n\ge4$.

Fix a point $x_0$ on $M$. For any $T>0$, there is a small $r_0$ such
that whenever $t\in [-T,0]$, $x\in B_t(x_0,r_0)$, we have
$$|R_m|(x,t)\le r_0^2 .$$
Let$A\rightarrow\infty$ in the Proposition 3.2, we get
$$(R+m\nu )(x,0)\ge -\frac{C_m}{T-0+\frac{1}{K_m}}\ge -\frac{C_m}{T} .$$

Since $C_m$ does not depend on T, we get that $(R+m\nu )(x,0)\ge 0$
for any m. This implies $\nu\ge0$, i.e., the curvature operator is
nonnegative.

\end{proof}

\begin{proposition}
Any gradient shrinking soliton(not necessarily having bounded
curvature) with vanishing Weyl tensor must have the following
properties:

(i) $Ric\ge0$;

(ii) $|R_{ijkl}|\le \exp\Big(a(d(p,x)^2+1)\Big)$ for some $a>0$ and
fixed point $p$.
\end{proposition}

\begin{proof}
(i) follows  from Corollary 3.3 immediately.

(ii) Clearly, it is suffices to prove the result under the condition
that the soliton is normalized. So $R+|{\nabla}f|^2-f=0$. Combine
the soliton equation $R_{ij}+\nabla_i\nabla_j f=\frac{1}{2} g_{ij}$
and (i), we get that $\nabla_i\nabla_j f\le\frac{1}{2} g_{ij}$.

For any point $x\in M$, let $\gamma (t):[0,d(p,x)]\rightarrow M$ be
the shortest normal geodesic connecting $p$ and $x$, and denote by
$h(t)=f(\gamma (t))$, then

$h''(t)=<{\nabla}f,\dot{\gamma}>'=\dot{\gamma}<{\nabla}f,\dot{\gamma}>=\nabla^2f(\dot{\gamma},\dot{\gamma})\le
\frac{1}{2}$,

By integrating above inequality we have
$$\arraycolsep=1.5pt\begin{array}{rcl}
f(x)&=& h(d(p,x))\\[4mm]
    &\le& \frac{1}{4}d(p,x)^2-<{\nabla}f, \dot{\gamma}>(0)d(p,x)-h(0)\\[4mm]
    &\le& \frac{1}{4}d(p,x)^2+|{\nabla}f|(p)d(p,x)+|f|(p).\\[4mm]
\end{array}$$
Since the right hand side of the above inequality just depends on
the information of $f$ at $p$, so by $R+|{\nabla}f|^2-f=0$, we have
that for some $a>0$, $R\le f \le \exp\Big(a(d(p,x)^2+1)\Big)$, hence
$|R_{ijkl}|\le \exp\Big(a(d(p,x)^2+1)\Big)$(because of the
nonnegativity of the curvature operator).

\end{proof}

Finally, by  \cite{NW} or \cite{PW}, we get the classification
theorem 1.2.

\bibliographystyle{amsplain}

\end{document}